\documentclass[12pt]{amsart}
\usepackage{amsmath,amsfonts,amstext,amsthm,epsfig}
\usepackage{xcolor}

\newtheorem{thm}{Theorem}

\newtheorem{prop}[thm]{Proposition}
\newtheorem{defn}[thm]{Definition}
\newtheorem{rem}{Remark}
\numberwithin{equation}{section}

\newcommand{\Real}{\mathbb R}
\newcommand{\eps}{\varepsilon}

\def\R{\mathbb R}

\begin{document}

\title[Well-posedness for hyperbolic equations]
{Well-posedness for hyperbolic equations whose coefficients lose regularity\\ at one point}%
\author{Daniele Del Santo}\author{Martino Prizzi}

\address{Daniele Del Santo, Universit\`a di Trieste, Dipartimento di
Matematica e Geoscienze,
Via Valerio 12/1, 34127 Trieste, Italy}%
\email{delsanto@units.it}%

\address{Martino Prizzi, Universit\`a di Trieste, Dipartimento di
Matematica e Geoscienze,
Via Valerio 12/1, 34127 Trieste, Italy}%
\email{mprizzi@units.it}%

\begin{abstract}
We prove some $C^\infty$ and Gevrey well-posedness results for hyperbolic
equations whose coefficients lose regularity at one point.
\end{abstract}
\maketitle

\section{Introduction}

In this paper we deal with the well-posedness of the Cauchy
problem for a linear hyperbolic operator whose coefficients depend only on
time. Namely, we consider the equation
\begin{equation}\label{eq1}
u_{tt}-\sum_{i,j=1}^n a_{ij}(t)u_{x_ix_j}=0
\end{equation}
in $[0,T]\times\R^n$, with initial data
\begin{equation}\label{eq2}
u(0,x)=u_0(x),\quad u_t(0,x)=u_1(x)
\end{equation}
in $\R^n$. The matrix $(a_{ij})_{i,j}$ is supposed to be real and symmetric. Setting
\begin{equation}\label{eq3}
a(t,\xi):=\sum_{i,j=1}^n a_{ij}(t)\xi_i\xi_j/|\xi|^2,\quad(t,\xi)\in[0,T]
\times(\R^n\setminus \{0\}),
\end{equation}
we assume throughout that $a(\cdot,\xi)\in L^\infty(0,T)$ for all $\xi\in\R^n\setminus\{0\}$.
Moreover, we suppose that the equation (\ref{eq1}) is strictly hyperbolic, i.e.
\begin{equation}\label{eq4}
\Lambda_0\geq a(t,\xi)\geq \lambda_0> 0
\end{equation}
for all $(t,\xi)\in[0,T]\times(\R^n\setminus \{0\})$.

It is a classical result that if the coefficients $a_{ij}(t)$'s
are real integrable functions, then the Cauchy problem (\ref{eq1}), (\ref{eq2})
is well posed in ${\mathcal A}'(\R^n)$, the space of real analytic functionals; moreover, if
the initial data vanish in a ball, then the solution vanishes in a cone, whose
slope depends on the coefficients $a_{ij}(t)$'s (see \cite[Theorems 1 and 3.a]{CDGS}).
On this basis,  various well-posedness results can be proved by mean of the Paley-Wiener theorem (in the version of
\cite[p. 517]{CDGS}, to which we refer here and throughout) and some energy estimates. If the
coefficients $a_{ij}(t)$'s are Lipschitz-continuous then the Cauchy problem
(\ref{eq1}), (\ref{eq2}) is well posed in Sobolev spaces. Relaxing this regulatity assumption, one has that if the $a_{ij}(t)$'s
are Log-Lipschitz-continuous or H\"older-continuous of index $\alpha$, then 
(\ref{eq1}), (\ref{eq2}) is well posed in $C^\infty$ or in the Gevrey space $\gamma^{(s)}$ for 
$s<{1\over 1-\alpha}$ respectively (see \cite[Theorem 3.b,c]{CDGS}). 
Suitable counterexamples show that in each case the regularity assumption on the $a_{ij}(t)$'s is sharp for the well posedness of (\ref{eq1}), (\ref{eq2}) in the corresponding function space.

It is a remarkable fact that in the above mentioned counterexamples the coefficients $a_{ij}(t)$'s are in fact $C^\infty$ for $t\not=0$, and each time the specific regularity fails only at $t=0$. In \cite{CDSK1} the authors showed that a control on the rate of the loss of Lipschitz regularity of the $a_{ij}(t)$'s as $t\to 0$ allows to recover well-posedness of (\ref{eq1}), (\ref{eq2}) in suitable function spaces. To be more specific, if the $a_{ij}(t)$'s are of class $C^1$ in $]0,T]$ and $|a_{ij}'(t)|\leq Ct^{-p}$, then (\ref{eq1}), (\ref{eq2}) is well posed in $C^\infty$ when $p=1$, and in the Gevrey space $\gamma^{(s)}$ for $s<{p\over p-1}$ when $p>1$. Concerning $C^\infty$ well-posedness, it was proved in \cite{CDSR}  that a control on the second derivative of the $a_{ij}$'s as $t\to0$ allows to relax slightly the growth assumption on the first derivative up to $|a_{ij}'(t)|\leq Ct^{-1}|\log t|$. 
In \cite{KR} some of the above results were extended to the case in which the coefficients $a_{ij}$'s depend also on the $x$ variable in $C^{\infty}$ fashion.

In this paper we consider non Lipschitz coefficients whose regularity is ruled by a modulus of continuity $\mu$, with a constant which blows up as $t\to 0$. More precisely, we assume that 
\begin{equation}
|a_{ij}(t+\tau)-a_{ij}(t)|\leq \frac C{\nu(t)}\mu(\tau),\quad 0\leq\tau\leq\tau_0,\quad t,t+\tau\in\,]0,T],
\end{equation}
where $\nu(t)^{-1}$ is possibly non integrable at $t=0$ and where $\mu$-continuity is possibly strictly weaker than Lipschitz continuity. We investigate how the interaction between $\nu$ and $\mu$ affects the well-posedness of (\ref{eq1}), (\ref{eq2}). 

In Section 2 we prove a technical regularization result for the coefficients $a_{ij}$'s.

In Section 3 we consider locally H\"older continuous coefficients satisfying
\begin{equation}|a_{ij}(t+\tau)-a_{ij}(t)|\leq \frac C{t^p}\tau^\alpha,\quad 0\leq\tau,\quad t,t+\tau\in\,]0,T]\end{equation}
with $0<\alpha<1$ and $p>1$, and we obtain well-posedness in the Gevrey space $\gamma^{(\sigma)}$ for $\sigma<\frac{p}{p-\alpha}$, a condition which fits perfectly with the ones of \cite{CDGS} and \cite{CDSK1}.

In Section 4 we consider the problem of $C^\infty$ well-posedness and we identify a precise relation between $\mu$ and $\nu$ which guarantees the latter. In particular we obtain well-posedness for coefficients satisfying 
\begin{equation}|a_{ij}(t+\tau)-a_{ij}(t)|\leq \frac{C}{t|\log t|}\,\frac{\tau|\log \tau|}{\log|\log\tau|},\quad 0\leq\tau,\quad t,t+\tau\in\,]0,T],\end{equation}
where one can easily see that $\nu(t)^{-1}$ is non integrable and  $\mu$-continuity is strictly weaker than Lipschitz continuity. Also in this situation the results fits with the ones contained in \cite{CDGS} and \cite{CDSK1} and contain them as particular cases.

\section{Approximation}
We begin by recalling the notion of {\it modulus of continuity}.
\begin{defn} Let $\tau_0>0$.
A function $\mu: [0,\, \tau_0]\to [0,\,+\infty[$ is a {\rm modulus of continuity} if it is continuous, concave, strictly increasing and $\mu(0)=0$. 
\end{defn}

Let $\mu$ be a modulus of continuity and let $a\colon [0,T]\to\Bbb R$  be a bounded function. Without loss of generality we can assume that $\tau_0\leq T$. We assume that
\begin{equation}|a(t+\tau)-a(t)|\leq \frac C{\nu(t)}\mu(\tau),\quad 0\leq\tau\leq\tau_0,\quad t,t+\tau\in\,]0,T],\end{equation}
where $\nu\colon\,]0,T]\to\,]0,+\infty[$ is a non-decreasing continuous function such that, for some $\kappa>0$,
\begin{equation}\nu(t/2)\geq \kappa \nu(t),\quad t\in\,]0,T].\label{doubling}\end{equation}

\begin{rem}
Condition (\ref{doubling}) is satisfied whenever $\nu$ is concave. Moreover, it is satisfied by $\nu(t)=t^p$ for every real exponent $p>0$. On the other hand, it is not satisfied if $\nu(t)$ tends to $0$ too fast as $t\to 0$, e.g. by $\nu(t)=e^{-1/t}$.
\end{rem}

Now let $0<\epsilon\leq\tau_0\leq T$ and define 
\begin{equation}
\tilde a_\epsilon(t):=\begin{cases}a(\epsilon)&\text{for $t\leq\epsilon$},\\a(t)&\text{for $\epsilon\leq t\leq T$},\\a(T)&\text{for $T\leq t$}.\end{cases}\label{reg1}
\end{equation}

Let $\rho\in C^\infty(\Bbb R)$ with ${\rm supp}\,\rho\subset[-1,1]$, $\rho(s)\geq 0$, $\int_{\Bbb R}\rho(s)\,ds=1$, set $\rho_\epsilon(s):=\frac1\epsilon\rho(\frac s\epsilon)$, and define
\begin{equation}a_\epsilon(t):=\int_{-\epsilon}^{\epsilon}\rho_\epsilon(s)\tilde a_\epsilon(t-s)\,ds, \quad t\in\Bbb R.\label{reg2}\end{equation}

We have the following

\begin{prop}\label{prop1}
Under the above hypotheses, there exist constants $C'$ and $C''>0$ such that, for $0<\epsilon\leq\tau_0$,
\begin{equation}
|a_\epsilon(t)-\tilde a_\epsilon(t)|\leq {C'}\min\left\{1,\frac1{\nu(t)}\mu(\epsilon)\right\},\quad t\in\,]0,T]\label{appr1}
\end{equation}
and
\begin{equation}
|a_\epsilon'(t)|\leq \frac{C''}\epsilon\min\left\{1,\frac1{\nu(t)}\mu(\epsilon)\right\},\quad t\in\,]0,T].\label{appr2}
\end{equation}
The constants $C'$ and $C''$ depend only on $C$, $\rho$, $\kappa$ and $\|a\|_{\infty}$.
\end{prop}
\begin{proof}
We have
\begin{multline*}
|a_\epsilon(t)-\tilde a_\epsilon(t)|=\left|\int_{t-\epsilon}^{t+\epsilon}\rho_\epsilon(t-s)(\tilde a_\epsilon(s)-\tilde a_\epsilon(t))\,ds\right|\\
\leq \int_{t-\epsilon}^{t+\epsilon}\rho_\epsilon(t-s)|\tilde a_\epsilon(s)-\tilde a_\epsilon(t)|\,ds.
\end{multline*}

If $t\geq2\epsilon$, then $t-\epsilon\geq t/2\geq\epsilon$, so $\nu(t-\epsilon)\geq\nu(t/2)\geq\kappa\nu(t)$. Therefore, we have
\begin{multline*}
|a_\epsilon(t)-\tilde a_\epsilon(t)|\leq \int_{t-\epsilon}^{t+\epsilon}\rho_\epsilon(t-s)| a(s)- a(t)|\,ds\\
\leq \int_{t-\epsilon}^{t+\epsilon}\rho_\epsilon(t-s)\frac C{\nu(t-\epsilon)}\mu(|s-t|)\,ds\\
\leq  \int_{t-\epsilon}^{t+\epsilon}\rho_\epsilon(t-s)\frac C{\nu(t/2)}\mu(|s-t|)\,ds\\
\leq  \int_{t-\epsilon}^{t+\epsilon}\rho_\epsilon(t-s)\frac{C/\kappa}{\nu(t)}\mu(|s-t|)\,ds\leq \frac{C/\kappa}{\nu(t)}\mu(\epsilon).
\end{multline*}
If $0<t\leq\epsilon$, then $\tilde a_\epsilon(s)=\tilde a_\epsilon(t)=a(\epsilon)$ for $s\leq\epsilon$, and therefore,
\begin{multline*}
|a_\epsilon(t)-\tilde a_\epsilon(t)|\leq \int_{\epsilon}^{t+\epsilon}\rho_\epsilon(t-s)| a(s)- a(\epsilon)|\,ds\\
\leq \int_{\epsilon}^{t+\epsilon}\rho_\epsilon(t-s)\frac C{\nu(\epsilon)}\mu(|s-\epsilon|)\,ds\\
\leq  \int_{\epsilon}^{t+\epsilon}\rho_\epsilon(t-s)\frac C{\nu(t)}\mu(\epsilon)\,ds\leq \frac{C}{\nu(t)}\mu(\epsilon).
\end{multline*}
If $\epsilon\leq t\leq 2\epsilon$, then $0\leq t-\epsilon\leq \epsilon\leq t$. Therefore, we have
\begin{multline*}
|a_\epsilon(t)-\tilde a_\epsilon(t)|\leq \int_{t-\epsilon}^{\epsilon}\rho_\epsilon(t-s)| a(\epsilon)- a(t)|\,ds+\int_{\epsilon}^{t+\epsilon}\rho_\epsilon(t-s)| a(s)- a(t)|\,ds\\
\leq \int_{t-\epsilon}^{\epsilon}\rho_\epsilon(t-s)\frac C{\nu(\epsilon)}\mu(|t-\epsilon|)\,ds+\int_{\epsilon}^{t+\epsilon}\rho_\epsilon(t-s)\frac C{\nu(\epsilon)}\mu(|t-s|)\,ds\\
\leq \frac{C}{\nu(\epsilon)}\mu(\epsilon)\leq  \frac{C}{\nu(t/2)}\mu(\epsilon)\leq \frac{C/\kappa}{\nu(t)}\mu(\epsilon).
\end{multline*}
The thesis follows setting $C':=\max\{C, C/\kappa, 2\|a\|_\infty\}$.

In order to estimate $a'_\epsilon$, we observe that 
\begin{multline*}
|a'_\epsilon(t)|=\left|\int_{t-\epsilon}^{t+\epsilon}\rho'_\epsilon(t-s)\tilde a_\epsilon(s)\,ds\right|=\left|\int_{t-\epsilon}^{t+\epsilon}\rho'_\epsilon(t-s)(\tilde a_\epsilon(s)-\tilde a_\epsilon(t))\,ds\right|\\
\leq \int_{t-\epsilon}^{t+\epsilon}|\rho'_\epsilon(t-s)||\tilde a_\epsilon(s)-\tilde a_\epsilon(t)|\,ds.
\end{multline*}
Then we procede as above, noticing that $\rho'_\epsilon(t)=\frac1{\epsilon^2}\rho'(\frac t\epsilon)$, and hence  $$\int_{t-\epsilon}^{t+\epsilon}|\rho'_\epsilon(t-s)|\,ds=\frac{\|\rho'\|_{L^1}}\epsilon.$$
The thesis follows setting $C'':=\|\rho'\|_{L^1}\max\{C, C/\kappa, \|a\|_\infty\}$.
\end{proof}

\section{Well posedness in Gevrey spaces}
In this section we shall prove that if the coefficients $a_{ij}$'s are locally H\"older continuous of exponent $\alpha$, with a H\"older constant which grows like $t^{-p}$ as $t\to0$, then the Cauchy problem (\ref{eq1}), (\ref{eq2}) is well posed in a suitable Gevrey space $\gamma^{(\sigma)}$, where $\sigma$ depends on $\alpha$ and $p$.

As we pointed out in the Introduction, since the coefficients $a_{ij}$'s
are real integrable functions, the Cauchy problem (\ref{eq1}), (\ref{eq2})
is well posed in
${\mathcal A}'(\R^n)$, the space of real analytic functionals (which have by definition compact support). Moreover, if
the
initial data vanish in a ball, then the solution vanishes in a cone, whose basis is the same ball and whose slope depends on the coefficients $a_{ij}$'s.  Therefore, it will be sufficient to
show that if $u_0$ and
$u_1$ belong to a a suitable Gevrey space $\gamma^{(\sigma)}$ and have compact support, then the corresponding solution
$u$ is not only in $W^{2,1}([0,T],{\mathcal A}'(\R^n))$, but it belongs
to the same Gevrey space in the $x$ variable for all $t\in[0,T$]. The result for initial data which do not have compact support follows by an exhaustion argument. Our main tools in the proof
will be the Paley-Wiener theorem and energy estimates.

\begin{thm}\label{th1}
Let $p> 1$ and $0<\alpha<1$, and assume that there exists a constant $C>0$ such that the function $a=a(t,\xi)$ defined by (\ref{eq3}) satisfies 
\begin{equation}|a(t+\tau,\xi)-a(t,\xi)|\leq \frac C{t^p}\tau^\alpha,\quad 0\leq\tau,\quad t,t+\tau\in\,]0,T]\label{control1}\end{equation}
for all $\xi\in \Bbb R^n\setminus\{0\}$.Then the Cauchy problem
(\ref{eq1}), (\ref{eq2}) is $\gamma^{(\sigma)}$-well-posed for
$1\leq\sigma<
\frac{p}{p-\alpha}$. 
\end{thm}

\begin{rem}
For a fixed $p>1$, passing to the limit as $\alpha\to1$ we regain the result of \cite{CDSK1}. In the same way, for a fixed $\alpha<1$, passing to the limit as $p\to 1$ we extend to $p=1$ the result of \cite{CDGS} which was valid only for $p<1$. The case $\alpha=1$, $p=1$ was considered in \cite{CDSK1} and will be reconsidered here in a more general context: in this case one has well posedness in $C^\infty$.
\end{rem}

\begin{rem} 

The result in Theorem \ref{th1} can be considered sharp in the following sense. Let $p_0> 1$ and $0<\alpha_0<1$. It is possible to construct a positive function $a\in C^\infty(]0,T])\cap C([0,T])$ such that
$$|a(t+\tau)-a(t)|\leq \frac C{t^{p_0}}\tau^{\alpha_0},\quad 0\leq\tau,\quad t,t+\tau\in\,]0,T],$$
and it is possible to construct two functions $u_0,\ u_1\in\gamma^{(s)}(\Real)$, for all $s>\frac{p_0}{p_0-\alpha_0}$ such that
the Cauchy problem 
$$
\left\{
\begin{array}{ll}
u_{tt}-a(t)u_{xx}=0\\[0.2cm]
u(0,x)=u_0(x),\ u_t(0,x)=u_1(x)
\end{array}\right.
$$
has no solution in $C^1([0,r[\,;\, D'^{(s)})$, for all $s>\frac{p_0}{p_0-\alpha_0}$ and for all $r>0$ (here $D'^{(s)} $ denotes the set of Gevrey-ultradistributions of index $s$). The construction of such a counterexample is exactly the same as that contained in Theorem 5 in \cite{CDSK1}.
 \end{rem}
\begin{rem} A result analogous to that of Theorem \ref{th1} can be proved if the singularity of the $a_{ij}$'s is located at $t=T$, with only minor obvious changes in the proof. As a consequence, the result is still valid if the coefficients have a finite number of singularities, where the loss of regularity is controlled as in (\ref{control1}).
\end{rem}

\begin{proof}[Proof of Theorem \ref{th1}]
We take the Fourier transform of $u$ with respect to $x$, and we denote it by $\hat u$. Equation
(\ref{eq1}) then transforms to
\begin{equation}\label{eq2.1}
\hat u_{tt}(t,\xi)+a(t,\xi)|\xi|^2\hat u(t,\xi)=0.
\end{equation}
Let $\epsilon$ be a positive parameter and for each
$\epsilon$ let $a_\eps\colon[0,T]\times(\R^n\setminus\{0\})\to\R$ be
defined according to (\ref{reg1})-(\ref{reg2}).

We define the {\em approximate energy} of $\hat u$ by
\begin{equation}\label{eq2.2}
E_\eps(t,\xi):=a_\eps(t,\xi)|\xi|^2|\hat u(t,\xi)|^2+|\hat u_t(t,\xi)|^2,\quad (t,\xi)\in
[0,T]\times(\R^n\setminus\{0\}).
\end{equation}
Differentiating $E_\eps$ with respect to $t$ and using (\ref{eq2.1}) we get
\begin{multline*}
E'_\eps(t,\xi)=a'_\eps(t,\xi)|\xi|^2|\hat u(t,\xi)|^2+
2a_\eps(t,\xi)|\xi|^2{\rm Re}(\hat u_t(t,\xi)\bar {\hat u}(t,\xi))\\
+2{\rm Re} (\hat u_{tt}(t,\xi)\bar {\hat u}_t(t,\xi))\\
\leq \left(\frac{|a'_\eps(t,\xi|}{a_\eps(t,\xi)}+
\frac{|a_\eps(t,\xi)-a(t,\xi)|}{a_\eps(t,\xi)^{1/2}}|\xi|\right)E_\eps(t,\xi).
\end{multline*}
By Gronwall's lemma we obtain
\begin{equation}\label{eq2.3}
E_\eps(t,\xi)\leq E_\eps(0,\xi)
\exp\left(\int_0^T\frac{|a'_\eps(t,\xi|}{a_\eps(t,\xi)}\,d t+
|\xi|\int_0^T\frac{|a_\eps(t,\xi)-a(t,\xi)|}{a_\eps(t,\xi)^{1/2}}\,d t\right)
\end{equation}
for all $t\in[0,T]$ and for all $\xi\in\R^n$, $|\xi|\geq1$.

By Proposition \ref{prop1} with $\mu(\tau)=\tau^\alpha$ and $\nu(t)=t^p$ we have 

\begin{multline*}
\int_0^T\frac{|a'_\eps(t,\xi|}{a_\eps(t,\xi)}\,d t+
|\xi|\int_0^T\frac{|a_\eps(t,\xi)-a(t,\xi)|}{a_\eps(t,\xi)^{1/2}}\,d t\\
\leq \int_0^T\frac{|a'_\eps(t,\xi|}{a_\eps(t,\xi)}\,d t+
|\xi|\int_0^T\left(\frac{|a(t,\xi)-\tilde a_\epsilon(t,\xi)|}{a_\eps(t,\xi)^{1/2}}+
\frac{|a_\eps(t,\xi)-\tilde a_\epsilon(t,\xi)|}{a_\eps(t,\xi)^{1/2}}\right)\,d t\\
\leq\int_0^{\epsilon^{\alpha/p}}\frac{C''}{\lambda_0\epsilon}\,dt+\int_{\epsilon^{\alpha/p}}^T\frac{C''}{\lambda_0\epsilon}
t^{-p}\epsilon^\alpha\,dt+\frac{2\Lambda_0}{\lambda_0^{1/2}}|\xi|\epsilon \\+ |\xi|\left(\int_0^{\epsilon^{\alpha/p}}\frac{C'}{\lambda_0^{1/2}}\,dt+\int_{\epsilon^{\alpha/p}}^T\frac{C'}{\lambda_0^{1/2}}
t^{-p}\epsilon^\alpha\,dt\right)\\
\leq M|\xi|\epsilon+ M\left(|\xi|+\frac1\epsilon\right)\left(\epsilon^{\alpha/p}+(\epsilon^{\alpha/p})^{1-p}\epsilon^\alpha\right)\\= M|\xi|\epsilon+ 2M\left(|\xi|+\frac1\epsilon\right)\epsilon^{\alpha/p},
\end{multline*}
where $M$ depends on $C'$, $C''$, $\lambda_0$, $\Lambda_0$, $\alpha$ and $p$.
Choosing $\epsilon=|\xi|^{-1}$ we get
\begin{equation}\label{mart1}
\left[\int_0^T\frac{|a'_\eps(t,\xi|}{a_\eps(t,\xi)}\,d t+
|\xi|\int_0^T\frac{|a_\eps(t,\xi)-a(t,\xi)|}{a_\eps(t,\xi)^{1/2}}\,d t\right]_{\epsilon=|\xi|^{-1}}\leq M+4M |\xi|^{\frac{p-\alpha}{p}}.
\end{equation}
Putting together (\ref{eq2.3}) and (\ref{mart1}) we get
\begin{equation}
E_{1/|\xi|}(t,\xi)\leq e^Me^{4M|\xi|^{\frac{p-\alpha}{p}}}E_{1/|\xi|}(0,\xi)\end{equation}
and, finally,
\begin{equation}
|\hat u_t(t,\xi)|^2+|\xi|^2|\hat u(t,\xi)|^2\leq \frac{e^M\Lambda_o}{\lambda_0}e^{4M|\xi|^{\frac{p-\alpha}{p}}}\left(|\hat u_t(0,\xi)|^2+|\xi|^2|\hat u(0,\xi)|^2\right)\label{mart2}
\end{equation}
Now if $u_0$, $u_1\in\gamma^{(\sigma)}\cap C^\infty_0$, the Paley-Wiener
theorem  ensures that there exist $K,\delta>0$ such that
\begin{equation}\label{eq2.4}
|\hat u(0,\xi)|^2+|\hat u_t(0,\xi)|^2\leq K\exp(-\delta|\xi|^{1/\sigma})
\end{equation}
for all $\xi\in\R^n$, $|\xi|\geq1$. It follows from (\ref{mart2}) that if $\sigma<p/(p-\alpha)$,
then there exist
$K',\delta'>0$ such that
\begin{equation}\label{eq2.5}
|\hat u(t,\xi)|^2+|\hat u_t(t,\xi)|^2\leq K'\exp(-\delta'|\xi|^{1/ \sigma})
\end{equation}
for all $t\in[0,T]$ and for all $\xi\in\R^n$, $|\xi|\geq1$ and, therefore, $u\in W^{2,1}([0,T],\gamma^{(\sigma)})$. The proof is complete.
\end{proof}

\section{Well posedness in $C^\infty$}
Let $\psi\colon [1,+\infty[\to\,]0,+\infty[$ be a strictly increasing continuous function, such that $\psi'$ is non-increasing and $e^r\psi'(r)$ is non-decreasing. Moreover, we assume that 
\begin{enumerate}
\item $\lim_{r\to\infty}\psi(r)=\chi$, $0<\chi\leq+\infty$;
\item $\lim_{r\to\infty}\psi'(r)=\eta$, $0\leq\eta<+\infty$; 
\end{enumerate}
We set
\begin{equation}
\nu(t):=\begin{cases}\frac{t}{\psi'(|\log t|)}&\text{for $0<t\leq e^{-1}$},\\\null&\null\\ \frac{e^{-1}}{\psi'(1)} &\text{for $e^{-1}\leq t$}.\end{cases}\label{blow}
\end{equation}
A direct computation shows that  $\nu$ is a non-decreasing continuous function and that
$\nu(t/2)\geq (1/2) \nu(t)$ for $t\in\,]0,T]$.
We define
\begin{equation}
\mu(\tau):=\frac{\tau|\log \tau|}{\psi(|\log \tau|)}\label{mod}
\end{equation}
and we assume that $\mu$ is strictly increasing and concave in $]0,\tau_0]$ for a suitable $\tau_0>0$, so it is a modulus of continuity.

\begin{thm}\label{th2}
Let $\nu=\nu(t)$ and $\mu=\mu(\tau)$ be as above, and assume that there exists a constant $C>0$ such that the function $a=a(t,\xi)$ defined by (\ref{eq3}) satisfies 
\begin{equation}|a(t+\tau,\xi)-a(t,\xi)|\leq \frac{C}{\nu(t)}\mu(\tau),\quad 0\leq\tau\leq\tau_0, \quad t,t+\tau\in\,]0,T],\label{control2}\end{equation}
for all $ \xi\in \Bbb R^n\setminus\{0\}$. Then the Cauchy problem
(\ref{eq1}), (\ref{eq2}) is well-posed in $C^\infty$.
\end{thm}

\begin{rem}Examples of functions satisfying all the above properties are $\psi(r)=1-e^{-\alpha r}$ with $0<\alpha\leq1$, $\psi(r)=1+\log r$ and $\psi(r)=r^\beta$ with $0<\beta\leq1$.
In particular, we have:
\begin{itemize} 
\item if $\psi(r)=r$ we have $\eta=1$ and $\chi=+\infty$ and we get $\mu(\tau)=\tau$ and $\nu(t)=t$, that is the situation considered in \cite{CDSK1};
\item if $\psi(r)=1-e^{-\alpha r}$ we have $\eta=0$ and $\chi=1$ and we get $\mu(\tau)=\tau|\log\tau|/(1-\tau^\alpha)$, which is equivalent to $\mu(\tau)=\tau|\log\tau|$, and $\nu(t)=\alpha t^{1-\alpha}$, that is a situation covered by the result of \cite{CDGS}, since $\nu(t)^{-1}$ is integrable;
\item if $\psi(r)=1+\log r$ or $\psi(r)=r^\beta$ with $0<\beta<1$, we have $\eta=0$ and $\chi=+\infty$, and we get $\mu(\tau)=\tau|\log\tau|/(1+\log|\log\tau|)$ or $\mu(\tau)=\tau|\log\tau|^{1-\beta}$. In both cases $\mu$-continuity is weaker than Lipschitz continuity. Moreover we have $\nu(t)=t|\log t|$ or $\nu(t)=t|\log t|^{1-\beta}$, so in both cases $\nu(t)^{-1}$ is not integrable.
\end{itemize}
The case in which $\mu(\tau)=\tau|\log\tau|$ and $\nu(t)^{-1}$ is not integrable is not covered by Theorem \ref{th2}, and we were not able to find a counterexample to $C^\infty$ well posedness either, so the question remains open. On the other hand, when $\mu(\tau)=\tau|\log\tau|$ and $\nu(t)=t$ by Theorem \ref{th1} we get authomatically $\gamma^{(\infty)}$ well posedness.
\end{rem} 
\begin{rem} A result analogous to that of Theorem \ref{th2} can be proved if the singularity of the $a_{ij}$'s is located at $t=T$, with only minor obvious changes in the proof. As a consequence, the result is still valid if the coefficients have a finite number of singularities, where the loss of regularity is controlled as in (\ref{control2}).
\end{rem}

\begin{proof}[Proof of Theorem \ref{th2}]
Like in the proof of Theorem \ref{th1}, we take the Fourier transform $\hat u$ of $u$. Equation
(\ref{eq1}) then transforms to
\begin{equation}\label{meq1}
\hat u_{tt}(t,\xi)+a(t,\xi)|\xi|^2\hat u(t,\xi)=0.
\end{equation}
For $0<\epsilon\leq\tau_1:=\min\{\tau_0, T,e^{-1}\}$ we define $a_\eps\colon[0,T]\times(\R^n\setminus\{0\})\to\R$ according to (\ref{reg1})-(\ref{reg2}). Again, we define an {\em approximate energy} of $\hat u$ by
\begin{equation}\label{meq2}
E_\eps(t,\xi):=a_\eps(t,\xi)|\xi|^2|\hat u(t,\xi)|^2+|\hat u_{t}(t,\xi)|^2,\quad (t,\xi)\in
[0,T]\times(\R^n\setminus\{0\}).
\end{equation}
Differentiating $E_\eps$ with respect to $t$ and using (\ref{meq1}) we get
\begin{multline*}
E'_\eps(t,\xi)=a'_\eps(t,\xi)|\xi|^2|\hat u(t,\xi)|^2+
2a_\eps(t,\xi)|\xi|^2{\rm Re}(\hat u_{t}(t,\xi)\bar {\hat u}(t,\xi))\\
+2{\rm Re} (\hat u_{tt}(t,\xi)\bar {\hat u}_t(t,\xi))\\
\leq \left(\frac{|a'_\eps(t,\xi|}{a_\eps(t,\xi)}+
\frac{|a_\eps(t,\xi)-a(t,\xi)|}{a_\eps(t,\xi)^{1/2}}|\xi|\right)E_\eps(t,\xi).
\end{multline*}
By Gronwall's lemma we obtain
\begin{equation}\label{meq3}
E_\eps(t,\xi)\leq E_\eps(0,\xi)
\exp\left(\int_0^T\frac{|a'_\eps(t,\xi|}{a_\eps(t,\xi)}\,d t+
|\xi|\int_0^T\frac{|a_\eps(t,\xi)-a(t,\xi)|}{a_\eps(t,\xi)^{1/2}}\,d t\right)
\end{equation}
for all $t\in[0,T]$ and for all $\xi\in\R^n$, $|\xi|\geq1$.
By Proposition \ref{prop1} with $\mu(\tau)$ and $\nu(t)$ given by (\ref{mod}) and (\ref{blow}), we have 
\begin{multline*}
\int_0^T\frac{|a'_\eps(t,\xi|}{a_\eps(t,\xi)}\,d t=\int_0^\epsilon\frac{|a'_\eps(t,\xi|}{a_\eps(t,\xi)}\,d t+\int_\epsilon^{e^{-1}}\frac{|a'_\eps(t,\xi|}{a_\eps(t,\xi)}\,d t+\int_{e^{-1}}^T\frac{|a'_\eps(t,\xi|}{a_\eps(t,\xi)}\,d t\\
\leq\frac{C''}{\lambda_0\epsilon}\left(\epsilon+\int_{\epsilon}^{e^{-1}}
\frac{\psi'(|\log t|)}{t}\frac{\epsilon|\log \epsilon|}{\psi(|\log \epsilon|)}\,dt+\int_{e^{-1}}^{T}\frac{\psi'(1)}{e^{-1}}\frac{\epsilon|\log \epsilon|}{\psi(|\log \epsilon|)}\,dt\right).\\
\end{multline*}
Since $$\frac{\psi'(|\log t|)}{t}=-\frac{d}{dt}\psi(|\log t|)$$ and $\psi(|\log\epsilon|)\geq\psi(|\log\tau_1|)$, we obtain
\begin{equation}
\int_0^T\frac{|a'_\eps(t,\xi|}{a_\eps(t,\xi)}\,d t\leq M''\left(1+|\log\epsilon|\right).
\end{equation}
On the other hand
\begin{multline*}
\int_0^T\frac{|a_\eps(t,\xi)-a(t,\xi)|}{a_\eps(t,\xi)^{1/2}}\,d t
=\int_0^T\left(\frac{|a(t,\xi)-\tilde a_\epsilon(t,\xi)|}{a_\eps(t,\xi)^{1/2}}+
\frac{|a_\eps(t,\xi)-\tilde a_\epsilon(t,\xi)|}{a_\eps(t,\xi)^{1/2}}\right)\,d t\\
\leq\frac{2\Lambda_0}{\lambda_0^{1/2}}\epsilon + \frac{C'}{\lambda_0^{1/2}}\left(\epsilon+\int_{\epsilon}^{e^{-1}}
\frac{\psi'(|\log t|)}{t}\frac{\epsilon|\log \epsilon|}{\psi(|\log \epsilon|)}\,dt+\int_{e^{-1}}^{T}\frac{\psi'(1)}{e^{-1}}\frac{\epsilon|\log \epsilon|}{\psi(|\log \epsilon|)}\,dt\right).
\end{multline*}
Arguing as above we, get
\begin{equation}\int_0^T\frac{|a_\eps(t,\xi)-a(t,\xi)|}{a_\eps(t,\xi)^{1/2}}\,d t\leq M'\epsilon\left(1+|\log\epsilon|\right).\end{equation}
Choosing $\epsilon=|\xi|^{-1}$ we get 
\begin{equation}\label{meq4}
\left[\int_0^T\frac{|a'_\eps(t,\xi|}{a_\eps(t,\xi)}\,d t+
|\xi|\int_0^T\frac{|a_\eps(t,\xi)-a(t,\xi)|}{a_\eps(t,\xi)^{1/2}}\,d t\right]_{\epsilon=|\xi|^{-1}}\leq M(\left(1+\log|\xi|\right)
\end{equation}
for $|\xi|\geq\tau_1^{-1}$.

Putting together (\ref{meq3}) and (\ref{meq4}) we get
\begin{equation}
E_{1/|\xi|}(t,\xi)\leq e^M |\xi|^M E_{1/|\xi|}(0,\xi)\end{equation}
and, finally,
\begin{equation}
|\hat u_{t}(t,\xi)|^2+|\xi|^2|\hat u(t,\xi)|^2\leq \frac{e^M\Lambda_o}{\lambda_0}|\xi|^M\left(|\hat u_{t}(0,\xi)|^2+|\xi|^2|\hat u(0,\xi)|^2\right).\label{meq5}
\end{equation}
Now if $u_0$, $u_1\in C^\infty_0$, the Paley-Wiener
theorem ensures that for all $\zeta>0$ there exists $K_\zeta>0$ such that
\begin{equation}\label{meq6}
|\hat u(0,\xi)|^2+|\hat u_{t}(0,\xi)|^2\leq K_\zeta |\xi|^{-\zeta}
\end{equation}
for all $\xi\in\R^n$, $|\xi|\geq\tau_1^{-1}$. It follows from (\ref{meq5}) that for all $\theta>0$
there exist
$K'_\theta>0$ such that
\begin{equation}\label{meq7}
|\hat u(t,\xi)|^2+|\hat u_{t}(t,\xi)|^2\leq K'_\theta |\xi|^{-\theta}
\end{equation}
for all $t\in[0,T]$ and for all $\xi\in\R^n$, $|\xi|\geq\tau_1^{-1}$, and therefore, $u\in W^{2,1}([0,T],C^\infty_0)$. The proof is complete. 

\end{proof}

\end{document}